\documentclass{article}
\usepackage[utf8]{inputenc}

\usepackage{amsmath,amsfonts,amssymb,amsthm,fancyhdr,bm,verbatim,hyperref,centernot}
\usepackage{fullpage}
\usepackage{mathtools}
\usepackage{todonotes}
\usepackage{amssymb}

\usepackage{tikz}
\usepackage{subcaption}
\tikzstyle{vertex}=[circle,fill=black,inner sep=2pt]

\makeatletter
\numberwithin{equation}{section}
\numberwithin{figure}{section}
\theoremstyle{plain}
\newtheorem{thm}{\protect\theoremname}
\theoremstyle{plain}

\theoremstyle{plain}

\newtheorem{prop}[thm]{\protect\propositionname}
\newtheorem{prob}[thm]{\protect\problemname}
\numberwithin{thm}{section}
\theoremstyle{definition}

\newcommand{\eps}{\varepsilon}
\newcommand{\ex}{\textnormal{ex}}

\makeatother

\usepackage{babel}
\providecommand{\corollaryname}{Corollary}
\providecommand{\lemmaname}{Lemma}
\providecommand{\theoremname}{Theorem}
\providecommand{\conjecturename}{Conjecture}
\providecommand{\propositionname}{Proposition}
\providecommand{\problemname}{Problem}

\title{Set-coloring Ramsey numbers via codes}
\author{David Conlon\thanks{Department of Mathematics, California Institute of Technology, Pasadena, CA 91125. Email: dconlon@caltech.edu. Research supported by NSF Award DMS-2054452.} \and
Jacob Fox\thanks{Department of Mathematics, Stanford University, Stanford, CA 94305. Email: jacobfox@stanford.edu. Research supported by a Packard Fellowship and by NSF Award DMS-1855635.} \and
Xiaoyu He\thanks{Department of Mathematics, Princeton University, Princeton, NJ 08544. Email: xiaoyuh@princeton.edu. Research supported by NSF Award DMS-2103154.} \and
Dhruv Mubayi\thanks{Department of Mathematics, Statistics and Computer Science, University of Illinois, Chicago, IL 60607. Email: mubayi@uic.edu. Research partially supported by NSF Awards DMS-1763317 and
DMS-1952767.} \and
Andrew Suk\thanks{Department of Mathematics, University of California at San Diego, La Jolla, CA 92093. Email: asuk@ucsd.edu. Research supported by an NSF
CAREER Award and NSF Award DMS-1952786.} \and 
Jacques Verstra\"ete\thanks{Department of Mathematics, University of California at San Diego, La Jolla, CA 92093. Email: jacques@ucsd.edu. Research supported by NSF Award DMS-1800332.}}
\date{}

\begin{document}

\maketitle

\begin{abstract} 
For positive integers $n,r,s$ with $r > s$, the set-coloring Ramsey number $R(n;r,s)$ is the minimum $N$ such that if every edge of the complete graph $K_N$ receives a set of $s$ colors from a palette of $r$ colors, then there is guaranteed to be a monochromatic clique on $n$ vertices, that is, a subset of $n$ vertices where all of the edges between them receive a common color. In particular, the case $s=1$ corresponds to the classical multicolor Ramsey number. We prove general upper and lower bounds on $R(n;r,s)$ which imply that $R(n;r,s) = 2^{\Theta(nr)}$ if $s/r$ is bounded away from $0$ and $1$. 
The upper bound extends an old result of Erd\H{o}s and Szemer\'edi, who treated the case $s = r-1$, while the lower bound exploits a connection to error-correcting codes. We also study the analogous problem for hypergraphs.
\end{abstract}

\section{Introduction}

The Ramsey number $R(n;r)$ is the minimum $N$ such that every $r$-coloring of the edges of the complete graph $K_N$ on $N$ vertices contains a monochromatic clique 
on $n$ vertices. The study of Ramsey numbers goes back more than a century, their first appearance arguably being in the work of Schur~\cite{Schur} from 1916 on Fermat's Last Theorem modulo a prime.

In the two-color case, classical results of Erd\H{o}s \cite{Erdos} and Erd\H{o}s--Szekeres \cite{ErdSzek} show that $2^{n/2} \leq R(n;2) \leq 4^n$. Despite concerted efforts, only lower-order improvements to these bounds have been made \cite{Conlon,Sah,Spencer} over the last seventy-five years and the exponential constants in the lower and upper bounds remain unchanged. 

For more colors, the best known bounds are of the form 
\begin{equation}\label{onecolor}2^{cnr} \leq R(n;r) \leq 2^{c'nr \log r}\end{equation}
for positive constants $c,c'$. Here the upper bound again follows from the method of Erd\H{o}s and Szekeres~\cite{ErdSzek}, while the lower bound was obtained by Lefmann \cite{Lef} through repeated application of the product inequality $R(n;r_1+r_2)-1 \geq \left(R(n;r_1)-1\right)\left(R(n;r_2)-1\right)$. 
For $r \geq 3$, the exponential constant $c$ in (\ref{onecolor}) was recently improved by Conlon and Ferber \cite{CoFe} through improving the bound for small $r$ and then applying Lefmann's product inequality, with further improvements subsequently made by Wigderson \cite{Wig} and by Sawin \cite{Saw}.

In this paper, we consider a natural generalization of the Ramsey problem where each edge is assigned a set of $s$ colors, instead of just one. Formally, for any positive integers $n,r,s$ with $r>s$, the \emph{set-coloring Ramsey number} $R(n;r,s)$ is the minimum positive integer $N$ such that in every edge-coloring $\chi: E(K_N)\rightarrow \binom{[r]}{s}$ there is a monochromatic $n$-clique, i.e., a set of $n$ vertices $v_1,\ldots ,v_n$ such that $\bigcap_{i<j}\chi(v_iv_j) \ne \emptyset$. In words, the set-coloring Ramsey number $R(n;r,s)$ is the minimum $N$ such that if every edge of $K_N$ receives a set of $s$ colors from a palette of $r$ colors, then there is guaranteed to be a monochromatic clique on $n$ vertices, that is, a copy of $K_n$ whose edges all share a common color. As a shorthand, it will be convenient for us to refer to such a set-coloring $\chi$ as an {\it $(r,s)$-coloring} of $K_N$.

The set-coloring Ramsey number and its special cases have been studied by many researchers for over half a century (see \cite{AEGM,BK22,BS18,EHR,ErSz,PSY,XSSL, ZXXC}). Indeed, note that when $s=1$, $R(n;r,1)=R(n;r)$ is just the classical Ramsey number. At the other extreme, when $s=r-1$, the problem was originally studied by Erd\H{o}s, Hajnal and Rado \cite{EHR} in 1965. Several years later, in 1972, Erd\H{o}s and Szemer\'edi~\cite{ErSz} proved that there are positive constants $c,c'$ such that 
\begin{equation} \label{eqn:ErSzem}
2^{cn/r} \leq R(n;r,r-1) \leq 2^{c'n\log r/r}.
\end{equation}
Comparing (\ref{onecolor}) and (\ref{eqn:ErSzem}), we see that at both  $s=1$ and $s=r-1$, the lower and upper bounds are off by a factor of $\log r$ in the exponent. Improving on the gap in either of these cases would be a major advance. 

In contrast, our main result, Theorem~\ref{corlow} below, allows us to determine the order of $R(n;r,s)$ up to an absolute constant in the exponent when $s/r$ is bounded away  from $0$ and $1$. More precisely, we have that
$R(n;r,s) = 2^{\Theta(nr)}$ whenever $\eps r < s < (1-\eps)r$ for a fixed $\eps > 0$. Somewhat surprisingly, this bound is of the same order as Lefmann's lower bound (\ref{onecolor}) for the $s = 1$ case, even though each edge now receives, say, half of the $r$ colors, instead of just one. 

An upper bound for $R(n;r,s)$ was proved in \cite{ZXXC} by extending the Erd\H{o}s--Szekeres \cite{ErdSzek} argument for $s=1$ to general $s$. For instance, a simple bound that can be deduced by this method is
\begin{equation}\label{simpleErSzek}R(n;r,s) \leq \frac{r}{r-s} \left(\frac{r}{s}\right)^{(n-2)r+1}.\end{equation}
However, if $s$ is close to $r$, the upper bound (\ref{simpleErSzek}) is not very good. Indeed, if $s=r-1$, the Erd\H{o}s--Szemer\'edi upper bound given in (\ref{eqn:ErSzem}) is considerably better. Here, we extend the Erd\H{o}s--Szemer\'edi argument to prove a better upper bound for general $s$. 
Both our upper and lower bounds are encapsulated in the following theorem.

\begin{thm}\label{corlow}
There exist constants $c,c'>0$ such that, for any integers $n, r, s$ with $n \geq 3$ and $r>s \geq 1$, 
$$a_r \cdot 2^{c n \max((r-s) r^{-1}, (r-s)^3 r^{-2})}
\leq R(n;r,s) \leq 2^{c'n(r-s)^2 r^{-1} \log\left(\frac{r}{\min(s,r-s)}\right)},$$
where $a_r$ is a constant depending only on $r$.
\end{thm}

In particular, setting $s = 1$ or $r-1$, 
we see that this result 
simultaneously generalizes the earlier results \eqref{onecolor} and \eqref{eqn:ErSzem}. 
Moreover, if $\eps r < s < (1-\eps) r$, then we obtain $R(n; r, s) = 2^{\Theta(nr)}$, as claimed above. 
In general, the lower bound for $\log_2 R(n; r,s)/((r-s)^2r^{-1} n)$ is at least a constant unless $s$ is close to $r$, when it can be as small as $r^{-1/2}$, while the upper bound is at most a constant unless $s$ is close to $1$ or $r$, when it can be as large as $\log r$.

The lower bound construction in Theorem~\ref{corlow}, which might be considered the main novelty of the paper, uses a product coloring together with the classical Gilbert--Varshamov lower bound on the size of the largest error-correcting code over a $q$-ary alphabet. 

Theorem~\ref{corlow} will be proved over the next two sections, beginning with this lower bound. In addition to this result, in Section~\ref{sec:codes}, we study $R(n;r,s)$ when $n$ is fixed and the color parameters $r$ and $s$ grow, showing that in this regime the problem is again tightly connected to error-correcting codes. Then, in Section~\ref{sec:hyper}, we study the natural hypergraph extension of $R(n;r,s)$. We  conclude with some problems and further remarks. For the sake of clarity of
presentation, we sometimes omit floor and ceiling signs when they are not essential.

\section{The lower bound in Theorem~\ref{corlow}} \label{sec:lower}

The goal of this section is to prove the lower bound in Theorem \ref{corlow}. We will need some notation pertaining to error-correcting codes. Let $A_q(m,d)$ be the maximum size of a code $C \subset [q]^m$ of length $m$ in which any two codewords have Hamming distance at least $d$, i.e., they differ in at least $d$ coordinates. Such a code is called a {\it $q$-ary code of length $m$ and distance $d$}. A basic fact from coding theory is that, for $d$ odd, $A_q(m,d)$ is the size of the largest $q$-ary code of length $m$ that corrects for $(d-1)/2$ errors. We deduce the lower bound in Theorem \ref{corlow} from Theorem~\ref{mainineq} below, which connects set-coloring Ramsey numbers to error-correcting codes.

Like Lefmann's product construction for $R(n;r)$, we construct an $(r,s)$-coloring with no monochromatic $K_n$ by taking a product of $(a,b)$-colorings with no monochromatic $K_n$, where $a$ and $b$ are much smaller than $r$ and $s$. We assign an edge between two vertices in this product the union of the set of colors on the edges in the coordinates where they differ. However, 
some edges may receive too few colors, so, instead of using the entire product set, we pass to an induced subgraph whose vertex set is an appropriate error-correcting code. 
Since any two points in the code differ in many coordinates,  this guarantees that there are many colors on each edge.

\begin{thm}\label{mainineq}
Let $a,b,d,m$ be positive integers with $b<a$ and $d<m$ and let $r=ma$ and $s=db$. Then, for $q=R(n;a,b)-1$,
$$R(n;r,s) > A_q(m,d).$$
\end{thm}

\begin{proof}
Consider an $(a,b)$-coloring $\chi:E(K_q)\rightarrow \binom{[a]}{b}$ with no monochromatic $K_n$. Such a coloring exists by the choice of $q=R(n;a,b)-1$. From $\chi$, we obtain a product set-coloring $\phi$ of the complete graph on $[q]^m$ as follows. The palette consists of the $r$ colors in $[a]\times [m]$ and an edge $(x_1,\ldots ,x_m)\sim (y_1,\ldots, y_m)$ is assigned the set $\bigcup_{i,x_i \not = y_i}(\chi(x_i,y_i)\times \{i\})$. Note that $\phi$ is not yet an $(r,s)$-coloring, because the number of colors on each edge can vary.

Also observe that in $\phi$, each edge gets a subset of the $r$ colors and there is no monochromatic $K_n$. For a given edge $(u,v)$, the number of colors $|\phi(u,v)|$ is exactly $b$ times the Hamming distance between $u$ and $v$. However, we would like for each edge to receive exactly $s$ colors. The edges that receive more than $s$ colors are easy to deal with, as we can simply take arbitrary subsets of their sets of colors of size $s$. The more substantial issue is that some edges might receive fewer than $s$ colors. In order to avoid this problem, we pick a code $C \subset [q]^{m}$ of maximum size and distance $d$ (so each edge between these vertices receives at least $db=s$ colors). 
By definition, such a code has size $A_q(m,d)$. Since each edge of the induced subgraph on $C$ receives at least $db=s$ colors and there is no monochromatic $K_n$, this completes the proof. 
\end{proof}

The Gilbert--Varshamov bound, due independently to Gilbert~\cite{Gil} and Varshamov~\cite{Var}, states that $A_q(m,d) \ge q^m/B(d-1)$, where $B(d-1)$ denotes the volume of a Hamming ball of radius $d-1$ in $[q]^m$. The exact formula for $B(d-1)$ is slightly unwieldy, so we will content ourselves with the simple estimate $B(d-1) \le (mq)^d$.  Together with the Gilbert--Varshamov bound, this gives $A_q(m,d) \ge m^{-d} q^{m-d}$.

With this estimate in hand, we are ready to prove the lower bound in Theorem~\ref{corlow}.

\begin{proof}[Proof of the lower bound in Theorem~\ref{corlow}.]
It suffices to show that, for $n$ sufficiently large in terms of $r$ and $s$,
\[
R(n;r,s) \ge 2^{cn \max((r-s)r^{-1}, (r-s)^3 r^{-2})}
\]
or, equivalently, that
\[
R(n;r,s) \ge 
\begin{cases}
2^{cn (r-s)/r} & \textnormal{ if } r-s \le 8 \sqrt{r} \\
2^{cn (r-s)^3/r^2} & \textnormal{ if } r-s > 8 \sqrt{r}.
\end{cases}
\]
Here $8$ is an arbitrary constant we chose for convenience.

As shown in  \cite[Theorem 4]{XSSL}, an application of the first moment method gives 
\begin{equation}\label{eq:first-moment}
R(n; a, b) \geq \left(a/b\right)^{(n-1)/2}(n!/a)^{1/n} \geq 2^{n(a-b)/3b},
\end{equation}
where the last inequality holds provided $b \geq a/2$ and $n! \geq a$ and uses $n \geq 3$ and $1+x \geq 2^x$ for $0 \leq x \leq 1$ with $x=(a-b)/b$. When $s > r - 8 \sqrt{r}$, we can apply (\ref{eq:first-moment}) directly with $a = r$ and $b = s$ to get $R(n; r, s) \geq 2^{n(r-s)/3s}\ge 2^{n(r-s)/6r}$, as desired.

Suppose now that 
$s \leq r - 8\sqrt{r}$ and let $a = \lfloor \frac{16r}{r-s} \rfloor \leq \frac{r-s}{4}$ and $b = a - 1$. If we choose the largest $r' \leq r$ which is a multiple of $a$ and the smallest $s' \geq s$ which is a multiple of $b = a-1$, then we can apply Theorem~\ref{mainineq} with $m = r'/a$ and $d = s'/b$ to conclude that 
$$R(n;r,s) \geq R(n;r',s') > A_q(m,d) \geq m^{-d} q^{m-d} = \left(\frac{a}{r'}\right)^{\frac{s'}{b}} (R(n;a,b)-1)^{\frac{r'}{a}-\frac{s'}{b}}.$$
But, since $s \leq s' + a$, 
$$\frac{r'}{a}-\frac{s'}{b} = \frac{r'}{a}-\frac{s'}{a-1} \geq \frac{r'}{a}-\frac{s'}{a} \left(1 + \frac{2}{a}\right) = \frac{r' - s'}{a}-\frac{2s'}{a^2} \geq \frac{r' - s'-2}{a}-\frac{2s}{a^2} \geq \frac{r - s}{2a} - \frac{2s}{a^2} \geq \frac{(r - s)^2}{64r}.$$
Hence, since $R(n; a, a-1) \geq 2^{n/6a} \geq 2^{n(r-s)/96r}$,
$$R(n;r,s) \geq \left(\frac{a}{r'}\right)^{\frac{s'}{b}}(R(n;a,a-1)-1)^{(r - s)^2/64r} \geq 2^{cn(r-s)^3/r^2}$$
for some $c > 0$, as required.
\end{proof}

\section{The upper bound in Theorem~\ref{corlow}}\label{sec:upper}

We now prove the upper bound on $R(n;r,s)$ claimed in Theorem \ref{corlow}. Observe that (\ref{simpleErSzek}) gives the claimed bound for $s \leq 9r/10$, so we may assume that $s > 9r/10$. The desired bound then follows from the next result. 

\begin{thm}
For any positive integers $r, s$ with $s > 9r/10$, 
$$R(n;r,s) \leq \left(\frac{r}{r-s}\right)^{500n(r-s)^2/r}.$$ 
\end{thm}
\begin{proof}
Let $N= \left(\frac{r}{r-s}\right)^{500n(r-s)^2/r}$ and suppose, for the sake of contradiction, that the edges of $K_N$ are each colored with $s$ colors from $[r]$ such that there is no monochromatic $K_n$.
We will describe a process, at each step of which some of the colors are turned on and the rest are turned off. We begin with all colors turned off. Once a color is turned on, it remains on. At each step $j$ of the process, we will have a set of remaining vertices $S_j$, chosen so that the sets $S_0 \supset S_1 \supset \cdots$ shrink at each step. We also define $\omega_i(S_j)$ to be the clique number of $S_j$ in color $i$ and keep track of the vector $(\omega_1(S_j),\ldots , \omega_r(S_j))$. To begin, we let $S_0=V(K_N)$, noting that $\omega_i(S_0) <n$ for each $i$.  Assuming $|S_j| > 10n$, we next describe how we find $S_{j+1} \subset S_j$. 

Suppose there is a color $i$ that is on in step $j$ with the density of edges in color $i$ within $S_j$ more than $1/2$. Then there is a vertex $v_j$ whose incident edges have density more than $1/2$ in color $i$. 
In this case, we let $S_{j+1}$ be the set of vertices that are adjacent to vertex $v_j$ in color $i$, so $|S_{j+1}| \geq |S_j|/2$. Since we may add $v_j$ to any clique in $S_{j+1}$ of color $i$ to form a larger clique of color $i$ in $S_j$, we find that $\omega_i(S_{j+1})\leq \omega_i(S_j) - 1$.

Otherwise, there is no color that is on at step $j$ which has edge density more than $1/2$. Letting $\eps = 1-s/r$, we then pick a color $i$ that is off in step $j$ such that the density of edges in color $i$ is at least $1-\eps$. 
Such a color exists since $\eps < 1/10$ and the average color density is $s/r=1-\eps$. Let $T \subset S_j$ be the set of vertices with edge density at least $1-2\eps$ in color $i$, so $|T| \geq |S_j|/2$. Pick a maximum monochromatic clique $Q \subset T$ in color $i$. As there is no monochromatic $K_n$, we have $|Q| < n < |S_j|/10 \leq |T|/5$ and each vertex in $Q$ has at most $2\eps|S_j|$ incident edges with vertices in $S_j$ that do not contain color $i$. Hence, each vertex in $Q$ has degree in color $i$ to $T \setminus Q$ at least $|T|-2\eps|S_j|-|Q| \geq (1-5\eps)|T \setminus Q|$. Therefore, the edge density in color $i$ from $Q$ to $T \setminus Q$ is at least $1-5\eps$. 

Let $U \subset T \setminus Q$ consist of those vertices whose degree in color $i$ is at least $(1-10\eps)|Q|$ to $Q$, so $|U| \geq (|T|-|Q|)/2 \geq |T|/4$. By the pigeonhole principle, there is a subset 
$Q' \subset Q$ with $|Q'| \leq 10\eps|Q|$ and a subset $S_{j+1} \subset U$ such that $|S_{j+1}| \geq |U|/{|Q| \choose 10\eps |Q|} \geq \eps^{10\eps n}|U| \geq \frac{1}{8} \eps^{10\eps n}|S_j|$ and each vertex in $S_{j+1}$ is such that its set of non-neighbors in color $i$ in $Q$ is a subset of $Q'$. We can make a monochromatic clique of color $i$ by combining $Q \setminus Q'$ and any monochromatic clique of color $i$ from $S_{j+1}$. By the choice of $Q$, this implies that $\omega_i(S_{j+1}) \leq 10\eps |Q| < 10 \eps n$. 
We then turn color $i$ on.

We next claim that at most $3(r-s)$ colors are ever on. Indeed, once $3(r-s)$ colors are on, no other color can turn on, as otherwise the sum of the edge densities of the colors would be at most 
$\frac{1}{2}(3(r-s))+\left((r-3(r-s)\right) < s$. However, this contradicts the fact, which follows since each edge has exactly 
$s$ colors on it, that the sum of the edge densities of the colors on any subset is $s$.  

At each step $j$ in which a color $i$ is turned on, we have $|S_j| \le 8 \eps^{-10\eps n} |S_{j+1}|$ and $\omega_i(S_{j+1})<10\eps n$, whereas at each step $j$ where we pass to a neighborhood of a color $i$ that is already on, we have $|S_j| \le 2 |S_{j+1}|$ and $\omega_i(S_{j+1}) \le \omega_i(S_j) - 1$. As at most $3(r-s)$ colors ever turn on in total, the latter type of step occurs at most $3(r-s)\cdot 10\eps n$ times. But, as $|S_0| = N \geq 10n \cdot 2^{3(r-s) \cdot 10\eps n}\left(8 \eps^{-10\eps n}\right)^{3(r-s)}$, this is a contradiction. 
\end{proof}

\section{More on set-coloring Ramsey numbers and codes} \label{sec:codes}

In this section, we explore another connection between set-coloring Ramsey numbers and error-correcting codes. 
Up to this point, we have focused on bounding $R(n; r, s)$ when $r$ and $s$ are fixed and $n$ is large. However, it is also interesting to look at the regime where $n$ is fixed. For example, when $n=3$ and $s=1$, $R(3;r,1)$ is simply the multicolor Ramsey number of a triangle.

We first observe that $R(n;r,s)=n$ if $(r-s){n \choose 2}<r$. Indeed, this is equivalent to $r(\binom{n}{2} - 1) < s \binom{n}{2}$, so, since each edge receives $s$ colors, 
some one of the $r$ colors must appear on  all $\binom{n}{2}$ edges. 
More generally, we have the following simple result, which connects set-coloring Ramsey numbers and Tur\'an numbers. Recall that, for a graph $H$ and a positive integer $N$, the Tur\'an number $\ex(N,H)$ is the maximum number of edges in an $H$-free graph on $N$ vertices.

\begin{prop}
If $N,n$ and $r>s$ are positive integers for which $\frac{s}{r} > \ex(N,K_n)/\binom{N}{2}$, 
then $R(n;r,s) \le N$. 
\end{prop}

\begin{proof}
Indeed, suppose $N < R(n;r,s)$, so there exists an $(r,s)$-coloring $\chi$ of $K_N$ with no monochromatic $K_n$. We bound the number of pairs $(e,i) \in E(K_N) \times [r]$ consisting of an edge $e$ and a color $i\in \chi(e)$ on that edge in two ways. First, as $|\chi(e)|=s$ for each $e$, the total number of such pairs is ${N \choose 2}s$. Second, as each color class is $K_n$-free, there are at most $\ex(N,K_n)$ edges of each color, giving an upper bound of $\ex(N,K_n)r$ on the number of pairs. Hence, $\ex(N,K_n)r \geq {N \choose 2}s$, which implies the proposition. 
\end{proof}

As a simple consequence of this proposition, we see that if $n \geq 3$ and $\eps > 0$ are fixed and $\frac{s}{r} \geq 1-\frac{1}{n-1} + \eps$, 
then $R(n;r,s)$ is at most a constant depending only on $n$ and $\eps$. On the other hand, for $n$ fixed and $\frac{s}{r} \leq 1-\frac{1}{n-1} - \eps$, we can use  Theorem~\ref{codebound} below to show that $R(n;r,s)$ grows at least exponentially in $r$. Thus, for $n$ fixed, there is a threshold 
for $s/r$ at $1-\frac{1}{n-1}$ where the set-coloring Ramsey number goes from being absolutely bounded to growing exponentially in $r$. 

Instead of looking at $R(n; r,s)$ directly, it will be useful to consider the variant $R'(n;r,s)$, defined to be the minimum $N$ such that in every $(r,s)$-coloring of $K_N$ there is a color $i$ for which the edges in color $i$ form a graph with chromatic number at least $n$. We clearly have $R(n;r,s) \geq R'(n;r,s)$ and we believe that this estimate
should be close to an equality when $s/r$ is close to the Tur\'an density $1-\frac{1}{n-1}$. Indeed, for the special case where $s = r-1$, this is confirmed in~\cite{AEGM}. The following result now establishes a tight connection 
between this variation of set-coloring Ramsey numbers and the size of error-correcting codes. 

\begin{thm} \label{codebound}
For all positive integers $n\ge 2$ and $r>s$,  $R'(n;r,s)=A_{n-1}(r,s)+1$. 
That is, the maximum $N$ for which there exists an $(r,s)$-coloring of $K_N$ such that each color class is $(n-1)$-partite is equal to the maximum size of a $q$-ary code with $q=n-1$ of length $r$ and distance $s$.  
\end{thm}

\begin{proof}
Consider a set-coloring of the edges of the complete graph on $[q]^r$, where a pair of vertices have color $i$ on their edge if they differ in coordinate $i$. Note that each color class is $q$-partite. Consider the induced subgraph on a code of size $A_q(r,s)$ of distance $s$. Each edge in this coloring of the complete graph on $A_q(r,s)$ vertices receives a set of at least $s$ colors from a set of $r$ colors and each color class is $q$-partite. Therefore, $R'(q+1;r,s) > A_q(r,s)$. 

For the other direction, let $N=R'(n;r,s)-1$. By definition, there exists an $(r,s)$-coloring of $K_N$ such that each color class is $(n-1)$-partite. For each color $i$, pick a partition $V(K_N)=V_{i,1} \cup \cdots \cup V_{i,n-1}$ of the vertex set into $n-1$ independent sets in color $i$. To each vertex $v \in V(K_N)$, assign $v$ the codeword $x(v)=(x_1(v),\ldots,x_r(v)) \in [n-1]^r$, where $x_i(v) = j$ if $v \in V_{i,j}$. Observe that no two vertices can have equal codewords as they must differ in at least $s$ coordinates and so the mapping $x:V(K_N) \rightarrow [n-1]^r$ is injective and its image $x(V(K_N))$ is an $(n-1)$-ary code of length $r$ with distance at least $s$. By the definition of $A_q(r,s)$, we find that $R'(n;r,s)-1 = N \leq A_{q}(r,s)$, completing the proof. 
\end{proof}

If we pick two random elements of $(n-1)^r$, the expected number of coordinates in which they differ is $(1- \frac{1}{n-1})r$. Standard tail estimates then imply that, for  any fixed $\eps >  0$, the probability two random elements differ in fewer than $(1- \frac{1}{n-1} -  \eps)r$ coordinates is at most $2^{-c r}$, where $c$ depends only on $n$ and $\eps$. In particular, this implies that we can pick $2^{c r/2}$ elements of $(n-1)^r$ such that the distance between any pair is at least $(1- \frac{1}{n-1} - \eps)r$.  That is, for $s \leq (1- \frac{1}{n-1} - \eps)r$, $A_{n-1}(r,s) \geq 2^{c r/2}$, so, by Theorem~\ref{codebound}, we get the same lower bound for $R'(n; r, s)$ and, hence, $R(n; r, s)$ in this range, confirming our earlier claim. 

Suppose now that we have a set $V$ of $N$ vertices, with $N$ a multiple of $n-1$, and there are $r$ partitions $P_1,\ldots,P_r$ of $V$ into $n-1$ sets of equal size such that, for each pair $u,v$ of distinct vertices, there are $s$ of the partitions $P_i$ for which $u$ and $v$ are in different parts. In this situation, we can produce an $(r,s)$-coloring of the edges of the complete graph on vertex set $V$, where a pair of distinct vertices $u,v$ receives the set of colors $i$ for which $u$ and $v$ are in different parts in $P_i$. Note that, in this set-coloring, each color class forms a Tur\'an graph with $n-1$ parts and so has the maximum possible number of edges for a $K_n$-free graph on $N$ vertices. If such partitions $P_1,\ldots,P_r$ exist, it therefore follows that $R(n;r,s)=R'(n;r,s)=A_{n-1}(r,s)+1=N+1$. 

We can build such partitions $P_1,\ldots,P_r$ using finite geometry, generalizing a construction of Alon, Erd\H{o}s, Gunderson and Molloy~\cite{AEGM} for the $s = r-1$ case. 
If $\mathbb{F}_q$ is the finite field of order $q$ for some  prime power $q$, then the number of $k$-dimensional subspaces in the $d$-dimensional affine space $\mathbb{F}_q^d$ is given by the Gaussian binomial coefficient (or $q$-binomial coefficient) $${d \choose k}_q = \prod_{i=0}^{k-1} \frac{(q^{d-i}-1)}{(q^{k-i}-1)}.$$
If we set $V=\mathbb{F}_q^d$ and let $S_1,\ldots,S_r$ be the set of all $k$-dimensional subspaces of $\mathbb{F}_q^d$, so that $r={d \choose k}_q$, we can define $r$ partitions $P_1,\ldots,P_r$ of $V$, where $x,y \in V$ are in the same part in $P_i$ if and only if $y-x \in S_i$. As there are $q^{d-k}$ distinct translates of $S_i$, each corresponding to a part of $P_i$, we get that each $P_i$ has $q^{d-k}$ parts and, hence, we can set $n=q^{d-k}+1$. Observe that each pair of distinct vertices is in the same number of affine subspaces of dimension $k$. Since the probability that two distinct random vertices are in the same part in $P_i$ is $\frac{q^k-1}{q^d-1}$, it follows that there are $s$ partitions for which they are in different parts, where 
$$s=\left(1-\frac{q^k-1}{q^d-1}\right)r.$$ 
Putting all this together, we have the following proposition.  

\begin{prop}
Let $k < d$ be positive integers and $q$ be a prime power. If $n=q^{d-k}+1$, $r={d \choose k}_q$ and $s=\left(1-\frac{q^k-1}{q^d-1}\right)r$, then  
$$R(n;r,s)=R'(n;r,s)=A_{n-1}(r,s)+1=q^d+1.$$
\end{prop}

While we expect that the bound $R(n;r,s) \geq R'(n;r,s) =A_{n-1}(r,s)+1$ is close to sharp for $s/r$ close to $1 - \frac{1}{n-1}$, it can be far from the truth in general. Indeed, the bound is only polynomial in $n$, as, by the Singleton bound~\cite{Sing}, $A_{n-1}(r,s) \leq (n-1)^{r-s+1}$, but in some other regimes we know that the set-coloring Ramsey number grows exponentially in $n$.

\section{Hypergraph set coloring} \label{sec:hyper}

One can also study the analogous problem for hypergraphs. Let us define $K_N^{(k)}$ to be the complete $k$-uniform hypergraph on $N$ vertices and an $(r,s)$-coloring of $K_N^{(k)}$ to be a function assignment $\chi:E(K_N^{(k)})\rightarrow \binom{[r]}{s}$. For positive integers $n, k, r, s$ with $n > k$ and $r > s$, the $k$-uniform set-coloring Ramsey number $R_k(n;r,s)$ is then the minimum positive integer $N$ such that in every $(r,s)$-coloring of $K_N^{(k)}$ there is a monochromatic $n$-clique, i.e., a set of $n$ vertices such that all $\binom{n}{k}$ edges in this set include a common color. A simple variant of an old result of Erd\H{o}s and Rado~\cite{ER} gives the following upper bound.

\begin{thm}  \label{thm:hyp2}
For any positive integers $n, k, r, s$ with $n > k\geq 3$ and $r > s$,
$$R_{k}(n; r,s)  \le  \binom{r}{s}^{{R_{k-1}(n-1; r,s) \choose k-1}}.$$
\end{thm}

In the other direction, there is a standard technique, 
the stepping-up lemma (see, for example,~\cite{GRS}), that allows one to take a bound for the $r$-color Ramsey number of graphs and lift it to a bound for the $2r$-color Ramsey number of $3$-uniform hypergraphs. This technique also applies in the set-coloring context and gives the following result.  We include its proof as a warm-up.

\begin{thm} \label{thm:hyp1}
If $R_2(n; r, s) > N$ for some positive integers $N, n, r, s$, then $R_3(n+1; 2r, s) > 2^N$.
\end{thm}
 
 \begin{proof}
Suppose that we have an $(r,s)$-coloring $\chi$ of the complete graph on vertex set $U = \{0,1,\ldots, N-1\}$ with no monochromatic $n$-clique. We will produce a $(2r,s)$-coloring $\phi$ of the complete $3$-uniform hypergraph on vertex set $V = \{0,1,\ldots, 2^N - 1\}$ with no monochromatic $(n+1)$-clique.

For $v \in V$, we write $v = \sum_{i =0}^{N-1}v(i)2^i$, where $v(i) \in \{0,1\}$.  Then, for any $u,v \in V$ with $u \neq v$, let $\delta(u,v)$ denote the largest $i\in \{0,1,\ldots ,N-1\}$ such that $u(i) \neq v(i)$.    It is easy to see that we have the following properties:

\medskip

\textbf{Property I}: For every triple $u < v < w$, $\delta(u,v) \neq \delta(v,w)$.

\textbf{Property II}: For $v_1 < \cdots < v_r$, $\delta(v_1,v_r) = \max_{1\leq j \leq r-1} \delta(v_j,v_{j+1})$.

\medskip

\noindent Given a triple $f = \{u,v,w\}$, where $u < v < w \in V$, set $\phi(f) = \chi(\delta(u,v),\delta(v,w)) \in {[r]\choose s}$ if $\delta(u,v) < \delta(v,w)$ and set $\phi(f) = \{x + r: x \in \chi(\delta(u,v),\delta(v,w))\}$ if $\delta(u,v) > \delta(v,w)$.  Hence,  $\phi(f) \in {[2r] \choose s}$ for every edge $f$.  For the sake of contradiction, suppose that there are vertices $v_1 < \cdots < v_{n + 1} \in V$ such that there is a common color in the coloring $\phi$ on the triples spanned by these vertices.  Let $\delta_i = \delta(v_i,v_{i + 1})$.  Since 
$\{x: x \in \chi(\delta_i,\delta_{i + 1}) \mbox{ for some } 1 \leq i \leq n\}$ and $\{x + r: x \in \chi(\delta_i,\delta_{i + 1}) \mbox{ for some } 1 \leq i \leq n\}$
are disjoint, the sequence $\{\delta_i\}$ must be monotone (i.e., $\delta_1   < \cdots < \delta_n$ or $\delta_1   >\cdots > \delta_n$).   By Properties I and II, this easily implies that there is a common color in the coloring $\chi$ on the set $\{\delta_1,\ldots, \delta_n\} \subset U$, which is a contradiction.\end{proof}

  Together with Theorem~\ref{corlow}, this implies that $R_3(n; r, s)$ is at least double exponential in $n$ provided $r \geq 2s + 2$.  For $k \geq 3$, it becomes a little easier to step-up from $k$-uniform to $(k+1)$-uniform hypergraphs. Following the stepping-up process described above, every $(k+1)$-tuple $\{v_1,\ldots, v_{k+1}\}$ with $v_1 < v_2 < \dots < v_{k+1}$ and $v_1,\ldots, v_{k+1} \in \{0,1,\ldots, 2^N - 1\}$ has a corresponding $k$-tuple $\{\delta_1,\ldots, \delta_k\}$ with $\delta_1,\ldots, \delta_k \in \{0,1,\ldots, N-1\}$, where $\delta_i = \delta(v_i,v_{i + 1})$. Because $k \geq 3$, we can now distinguish between whether the sequence $\{\delta_i\}$ is monotone or non-monotone.  If $\{\delta_i\}$ is monotone, we assign the $(k+1)$-tuple $\{v_1,\ldots, v_{k + 1}\}$ the same color as the $k$-tuple $\{\delta_1,\ldots, \delta_k\}$.  If $\{\delta_i\}$ is non-monotone, we distinguish between whether the first non-monotone triple in the sequence $\{\delta_i\}$ is a local maximum or a local minimum (i.e., $\delta_i < \delta_{i + 1} > \delta_{i + 2}$ or $\delta_j > \delta_{j + 1} < \delta_{j + 2}$), in the first case giving the corresponding $(k+1)$-tuple one particular fixed set of $s$ colors and in the second case a totally disjoint fixed set of $s$ colors. In particular, both cannot occur in a monochromatic clique. 
  For this to work, there must be at least $2s$ colors overall, that is, we need $r \geq 2s$. 
  A careful analysis now gives the following result.

\begin{thm} \label{thm:stepup3}
Fix  $k \ge 3$. If $1 \le s \le r/2$ and $R_k(n; r, s) > N>0$ for some positive integers $N, n, r, s$, then $R_{k+1}(2n-1; r, s) > 2^N$.
\end{thm}

Thus, unlike for ordinary Ramsey numbers, where a double-exponential lower bound for $R_3(n; r, 1)$ automatically implies a triple-exponential lower bound for $R_4(n; r, 1)$, it is not always the case that a double-exponential lower bound for $R_3(n; r, s)$ implies a triple-exponential lower bound for $R_4(n; r, s)$.  Instead, Theorem~\ref{thm:stepup3} only says that we can apply the stepping-up method as long as $s \le r/2$. 
The next theorem shows that, provided we are willing to pay a small price in the top exponent, 
we can actually go as far as $s \le 2r/3$. 

\begin{thm} \label{thm:stepup4}
If $1\le s \le 2r/3$ and $R_3(n; r, s) > N>0$ for some positive integers $N, n, r, s$,  then $R_{4}(2n^2; r, s) > 2^N$.
\end{thm}
\begin{proof}

Suppose we have an $(r,s)$-coloring $\chi$ of the complete 3-uniform hypergraph with vertex set $U = \{0,1,\ldots, N-1\}$ with no monochromatic $n$-clique. We will produce an $(r,s)$-coloring $\phi$ of the complete 4-uniform hypergraph on vertex set $V = \{0,1,\ldots, 2^N - 1\}$ with no monochromatic $2n^2$-clique.  Just as in the proof of Theorem~\ref{thm:hyp1}, for $v \in V$, we write $v = \sum_{i =0}^{N-1}v(i)2^i$, where $v(i) \in \{0,1\}$.  For any $u,v \in V$ with $u \neq v$, let $\delta(u,v)$ denote the largest $i\in \{0,1,\ldots ,N-1\}$ such that $u(i) \neq v(i)$.  Once again, Properties I and II hold.

Choose three $s$-subsets $A, B, C$ of $[r]$ with the property that $A \cap B \cap C = \emptyset$. Since  $s \le 2r/3$, this is possible. We are now ready to define our coloring $\phi$. Given a 4-tuple $f=\{v_1, \ldots, v_4\}$ with $v_1<v_2<v_3<v_4 \in V$, let $\delta_i=\delta(v_i, v_{i+1})$ and color the edges as follows: 
\begin{itemize}
 \item   If the sequence $\{\delta_i\}$ is monotone, then $\phi(f) = \chi(\delta_1, \delta_2, \delta_3)$.
 \item If $\delta_3 > \delta_1 > \delta_2$ then $\phi(f) = A$.
 
 \item If $\delta_1>\delta_3>\delta_2$ then $\phi(f) = B$. 
 
 \item If $\delta_1<\delta_2> \delta_3$ then $\phi(f) = C$.
\end{itemize}

\noindent See Figure \ref{exfig}.

\begin{figure}[h]
\begin{subfigure}{.33\textwidth}
\centering
\begin{tikzpicture}[scale=.5]
\node at (1,0) {$\delta_3$};
\node at (2,0) {$\delta_1$};
\node at (3,0) {$\delta_2$};

\draw (-.5,-1) node {$v_1$:} (1,-1) node {0} (2,-1) node {0} (3,-1) node {*};
\draw (-.5,-2) node {$v_2$:} (1,-2) node {0} (2,-2) node {1} (3,-2) node {0};
\draw (-.5,-3) node {$v_3$:} (1,-3) node {0} (2,-3) node {1} (3,-3) node {1};
\draw (-.5,-4) node {$v_4$:} (1,-4) node {1} (2,-4) node {*} (3,-4) node {*};
\end{tikzpicture}
\caption{$\phi(f) = A$}
\end{subfigure} \begin{subfigure}{.33\textwidth}
\centering
\begin{tikzpicture}[scale=.5]
\node at (1,0) {$\delta_1$};
\node at (2,0) {$\delta_3$};
\node at (3,0) {$\delta_2$};

\draw (-.5,-1) node {$v_1$:} (1,-1) node {0} (2,-1) node {*} (3,-1) node {*};
\draw (-.5,-2) node {$v_2$:} (1,-2) node {1} (2,-2) node {0} (3,-2) node {0};
\draw (-.5,-3) node {$v_3$:} (1,-3) node {1} (2,-3) node {0} (3,-3) node {1};
\draw (-.5,-4) node {$v_4$:} (1,-4) node {1} (2,-4) node {1} (3,-4) node {*};
\end{tikzpicture}
\caption{$\phi(f) = B$}
\end{subfigure}
\begin{subfigure}{.33\textwidth}
\centering
\begin{tikzpicture}[scale=.5]
\node at (1,0) {$\delta_2$};
\node at (2,0) {$\delta_1$};
\node at (3,0) {$\delta_3$};

\draw (-.5,-1) node {$v_1$:} (1,-1) node {0} (2,-1) node {0} (3,-1) node {*};
\draw (-.5,-2) node {$v_2$:} (1,-2) node {0} (2,-2) node {1} (3,-2) node {*};
\draw (-.5,-3) node {$v_3$:} (1,-3) node {1} (2,-3) node {$x$} (3,-3) node {0};
\draw (-.5,-4) node {$v_4$:} (1,-4) node {1} (2,-4) node {$x$} (3,-4) node {1};
\end{tikzpicture}
\, \, 
\begin{tikzpicture}[scale=.5]
\node at (1,0) {$\delta_2$};
\node at (2,0) {$\delta_3$};
\node at (3,0) {$\delta_1$};

\draw  (1,-1) node {0} (2,-1) node {0} (3,-1) node {0};
\draw  (1,-2) node {0} (2,-2) node {0} (3,-2) node {1};
\draw  (1,-3) node {1} (2,-3) node {0} (3,-3) node {*};
\draw  (1,-4) node {1} (2,-4) node {1} (3,-4) node {*};
\end{tikzpicture}

\caption{$\phi(f) = C$}

\end{subfigure}
\caption{Examples of $v_1 < \cdots < v_4$ and $\delta_i = \delta(v_i,v_{i + 1})$.  Each $v_i$ is represented in binary with the leftmost entry corresponding to the most significant bit. A $*$ indicates that the corresponding bit can be  0 or 1 independently of all other $*$'s.}
\label{exfig}
\end{figure}

Let $v_1<\cdots < v_{2n^2}$ be vertices in $V$ and suppose, for the sake of contradiction, that there is a common color in the coloring $\phi$ on all the edges spanned by these vertices. Let $\delta_i = \delta(v_i, v_{i+1})$ and suppose that $\delta_{i_0} = \max_i \delta_i$.  By Property II, $\delta_{i_0} = \max_{i,j} \delta(v_i, v_j)$. In other words, the maximum $\delta$ value is achieved by two consecutive vertices. 

Our main claim is that $\min \{i_0, 2n^2-i_0\} < n$. To see this, suppose, for the sake of contradiction, that $i_0 \geq n$ and $2n^2-i_0 \geq n$. If the sequence $\{\delta_j\}_{j=1}^{i_0}$ is monotone increasing, then, by Properties I and II, there is no common color on the vertex set $\{v_1, \ldots, v_n, \ldots, v_{i_0}, v_{i_0 +1}\}$ in the coloring $\phi$ as there is no common color on the vertex set $\{\delta_1, \ldots, \delta_n\}$ in the coloring $\chi$. A similar argument shows that $\{\delta_j\}_{j=i_0}^{2n^2-1}$ is not monotone decreasing. Consequently, there exist $1\le j < i_0 < k < 2n^2-1$ such that 
$$\delta_j > \delta_{j+1}  < \delta_{i_0} > \delta_{k} < \delta_{k+1}.$$
For $f_1=\{v_j, v_{j+1}, v_k, v_{k+1}\}$, the corresponding sequence of $\delta$'s is 
$$(\delta_j, \delta(v_{j+1}, v_{k}), \delta_k).$$ 
Since  $\delta(v_{j+1}, v_{k})= \delta_{i_0}$ by Property II, we have the inequalities $\delta_j<\delta(v_{j+1}, v_{k})>\delta_k$ and, consequently, $\phi(f_1) = C$. For $f_2=\{v_j, v_{j+1}, v_{j+2}, v_{i_0+1}\}$, the corresponding sequence of $\delta$'s is 
$$(\delta_j, \delta_{j+1}, \delta(v_{j+2}, v_{i_0+1})).$$ 
Since  $\delta(v_{j+2}, v_{i_0+1})= \delta_{i_0}$, we have the inequalities $\delta(v_{j+2}, v_{i_0+1})> \delta_j >\delta_{j+1}$ and, consequently, $\phi(f_2) = A$. For $f_3=\{v_{i_0}, v_{k}, v_{k+1}, v_{k+2}\}$, the corresponding sequence of $\delta$'s is 
$(\delta_{i_0}, \delta_{k}, \delta_{k+1})$ and 
   $\delta_{i_0}> \delta_{k+1} >\delta_{k}$  yields $\phi(f_3) = B$. Since $f_1, f_2, f_3$ receive the color sets $C, A, B$, respectively, and $A \cap B \cap C = \emptyset$, we have a contradiction.

  We have shown that $\min \{i_0, 2n^2-i_0\} < n$. Color $i_0$ black if $i_0 < n$ and white if $2n^2-i_0 < n$. Assume that $i_0$ is black (a similar argument works if it is white). We now consider the vertex set $\{v_{i_0+1}, v_{i_0+2}, \ldots, v_{2n^2}\}$ and define $\delta_{i_1} = \max_{i_0<i \le 2n^2-1} \delta_i$ and repeat the procedure just described to color $i_1$ black or white. In this way we obtain a sequence $i_0, i_1, i_2, \ldots, i_j, \ldots, i_{2n}$, where $i_j$ is defined from a vertex subset of size $2n^2-(n-1)j$, together with a black/white coloring of the elements in the sequence.  By the pigeonhole principle, at least $n+1$ of the $i_j$'s, say $i_{j_1}, i_{j_2}, \ldots, i_{j_{n+1}}$,  have the same color.  We consider the set of vertices $\{v_{i_{j_1}}, v_{i_{j_2}}, \ldots, v_{i_{j_{n+1}}}\}$. By construction and Properties I and II, the corresponding sequence
  $\delta(v_{i_{j_1}}, v_{i_{j_2}}), \ldots, \delta(v_{i_{j_{n}}}, v_{i_{j_{n+1}}})$ is monotone and, in  particular, consists of $n$ distinct numbers. Consequently, it has no common color among its triples in the coloring $\chi$ and, therefore, $\{v_{i_{j_1}}, v_{i_{j_2}}, \ldots, v_{i_{j_{n+1}}}\}$ has no common color among its 4-tuples in the coloring $\phi$, a contradiction.  \end{proof}
We note, in particular, that Theorem~\ref{thm:stepup4} implies that $R_4(n; 3,2)$ is at least double exponential in 
$\sqrt{n}$. Moreover, a double-exponential lower bound for $R_3(n; 3, 2)$, which remains an open problem, would imply a triple-exponential lower bound for $R_4(n; 3,2)$.

\section{Concluding remarks}

One curiosity of Theorem~\ref{corlow} is that though the upper and lower bounds agree up to a constant in the exponent when $s/r$ is bounded away from $0$ and $1$, they diverge when $s = o(r)$ or $r - s = o(r)$. 
Indeed, when $r - s = o(r)$, they can diverge significantly. 
For instance, we only have that
\[2^{c n/\sqrt{r}} \leq R(n; r, r - \sqrt{r}) \leq 2^{c' n \log r}\]
and it would be interesting to decide which of these bounds is closer to the truth. We leave this as an open problem.

\begin{prob}
Improve the bounds for $R(n; r, r - \sqrt{r})$.
\end{prob}

Regarding hypergraphs, Theorem~\ref{thm:hyp1} 
leaves open the problem of whether $R_3(n; r, s)$ is double exponential when $r \leq 2s + 1$. We highlight the second case of interest (the first case, where $r = 3$ and $s = 1$, is already notorious).

\begin{prob}
Show that $R_3(n; 5, 2)$ is double exponential in $n$.
\end{prob}

The case $R_3(n; 4,2)$ also seems particularly interesting because of its connections to classical problems on hypergraph Ramsey numbers. 
Indeed, for positive integers $n,k,r,s,t$, note that 
$R_k(n;tr,ts) \ge R_k(n;r,s)$, since, given an $(r,s)$-coloring with no monochromatic $K_n$, we can replace each color by a set of $t$ new colors to find a $(tr, ts)$-coloring with no monochromatic $K_n$.  In the other direction, we have $R_k(n;r,s) \le R_k(n;r-1,s-1)$, since, given an $(r,s)$-coloring with no monochromatic $K_n$, we can produce an $(r-1, s-1)$-coloring with the same property by deleting color $r$ from each edge in which it appears and deleting an arbitrary color from every other edge. Therefore,
$$R_3(n,n) := R_3(n; 2,1) \le R_3(n; 4,2) \le R_3(n; 3,1) =: R_3(n,n,n).$$
That is, the problem of determining $R_3(n; 4,2)$ is intermediate between those of determining $R_3(n,n)$ and $R_3(n,n,n)$, so progress on understanding the former function may shed some light on the latter ones.

It is possible to improve Theorems~\ref{thm:stepup3} and \ref{thm:stepup4} further still for $k \geq 4$, in that the necessary inequality between $s$ and $r$ when stepping-up from $k$-uniform to $(k+1)$-uniform hypergraphs can be improved to $s \le (1-c_k)r$ where $c_k$ tends to $0$ as $k$ tends to infinity. 
However, it remains the case that we cannot show 
$R_{k+1}(n; r,  s)$ is at least exponential in $R_k(n; r, s)$ for all choices of $r$ and $s$. All of the difficulties are inherent in the following general question.

\begin{prob}
Determine the tower height of $R_k(n; r, r-1)$ for all $k \ge 3$ and $r \ge 2$.
\end{prob}

Since this problem, which includes the famous question of estimating $R_3(n,n)$ as a special case, is likely to be very difficult, a simpler problem might be to show that there exists a bounded $c$, independent of $r$, such that $R_k(n; r, r-1)$ has tower height at least $k-c$ for all $k \geq 3$ and $r \geq 2$.

\vspace{3mm}
\noindent
{\bf Acknowledgements.} This research was initiated during a visit to the American Institute of Mathematics under their SQuaREs program. We would like to thank Noga Alon, Matija Buci\'c and Yuval Wigderson for helpful conversations.

\end{document}